\documentclass[12pt]{amsart}

\usepackage{amssymb}

\newtheorem{thmA}{Theorem}

\newtheorem{theorem}{Theorem}[section]
\newtheorem{thm}[theorem]{Theorem}
\newtheorem{lemma}[theorem]{Lemma}

\newtheorem{prop}[theorem]{Proposition}

\theoremstyle{definition}
\newtheorem{definition}[theorem]{Definition}

\theoremstyle{remark}
\newtheorem{remark}[theorem]{Remark}

\numberwithin{equation}{section}

\def\<{\langle}
\def\>{\rangle}

\def\C{\mathcal C}

\def\-{\overline}

\def\-{\underline}
\def\Z{\mathbb Z}

\def\G{\Gamma}

\def\C{\mathcal{C}}

\catcode`\@=11
\def\serieslogo@{\relax}
\def\@setcopyright{\relax}
\catcode`\@=12

\begin{document}

\title{Conjugacy separability of 1-acylindrical graphs of free groups}

\author[Owen Cotton-Barratt]{Owen Cotton-Barratt}
\address{Mathematical Institute, 24-29 St Giles', Oxford OX1 3LB, UK}
\email{cotton-b@maths.ox.ac.uk}

\author[Henry Wilton]{Henry Wilton}
\address{Department of Mathematics, 1 University Station C1200, Austin, TX 78712, USA}
\email{henry.wilton@math.utexas.edu}

\date{30 May 2009}



\thanks{}

\begin{abstract}
We use the theory of group actions on profinite trees to prove that the fundamental group of a finite, 1-acylindrical graph of free groups with finitely generated edge groups is conjugacy separable. This has several applications: we prove that positive, $C'(1/6)$ one-relator groups are conjugacy separable; we provide a conjugacy separable version of the Rips construction; we use this latter to provide an example of two finitely presented, residually finite groups that have isomorphic profinite completions, such that one is conjugacy separable and the other does not even have solvable conjugacy problem.
\end{abstract}

\maketitle

\section{Introduction}\label{s: Introduction}

A group $G$ is called \emph{conjugacy separable} if every conjugacy class is closed in the profinite topology on $G$.  This is a natural strengthening of the notion of residual finiteness, which asserts that the conjugacy class $\{1\}$ is closed.  A longstanding open problem in geometric group theory asks whether every (word-)hyperbolic group is residually finite, and it is equally natural to extend this question to ask which hyperbolic groups are conjugacy separable.

In this paper we prove that every member of a large class of hyperbolic groups is conjugacy separable.  A graph of groups is \emph{$k$-acylindrical} if only the trivial element of the fundamental group fixes a subset of the Bass--Serre tree of diameter greater than $k$.  In \cite{Wise2}, Wise studied finite, 1-acylindrical graphs of free groups and proved that their fundamental groups are residually finite (as long as the edge groups are finitely generated).  Our main theorem extends \cite{Wise2} to obtain the much stronger conclusion that the fundamental groups of such graphs of groups are conjugacy separable.

\begin{thmA}\label{t: Main Theorem}
If $\mathcal{G}$ is a finite, 1-acylindrical graph of free groups with finitely generated edge groups then $\pi_1(\mathcal{G})$ is conjugacy separable.
\end{thmA}

If $\pi_1(\mathcal{G})$ is finitely generated then it is word-hyperbolic by the Combination Theorem \cite{bestvina_combination_1992}. The extra techniques we use to deduce conjugacy separability are quite different from those of \cite{Wise2}.  Rather, we appeal to the technology of group actions on profinite trees, as developed in \cite{zalesskii_fundamental_1989}.

Wise was able to apply \cite{Wise2} to a wide variety of interesting classes of groups \cite{Wise1,Wise3}, and Theorem \ref{t: Main Theorem} can be applied in many of these situations.  A little care is needed here: residual finiteness passes to finite extensions, but this is not necessarily the case for conjugacy separability.  Nevertheless, under certain mild hypotheses we are able to sidestep this technicality (see Lemma \ref{separabilitypasses}).  Here follows a list of some applications of Theorem \ref{t: Main Theorem}.

The Rips construction is a powerful technique for constructing pathological hyperbolic groups \cite{Rips}.  Given any finitely presented group as input, the Rips construction outputs a hyperbolic group; pathologies of the input group often translate to new pathologies of the output group.  In \cite{Wise1}, Wise provided a version of the Rips construction in which the output group is residually finite.  It follows from Theorem \ref{t: Main Theorem} that the output of Wise's construction is actually conjugacy separable.

\begin{thmA}\label{t: Rips construction}
There exists an algorithm that, given a finitely presented group $Q$ as input, outputs a short exact sequence $1\rightarrow N\rightarrow \Gamma\rightarrow Q \rightarrow 1$ such that:
\begin{enumerate}
\item $N$ is finitely generated; and
\item $\Gamma$ is hyperbolic, torsion-free, and conjugacy separable.
\end{enumerate}
\end{thmA}

Our original motivation for proving Theorem \ref{t: Main Theorem} is the following extension of the results of \cite{BridGrun}.  It provides a counterexample to the na\"{i}ve conjecture that conjugacy separability is a property of the profinite completion for residually finite, finitely presented groups.

\begin{thmA}\label{t: Pair}
There exists a finitely presented, conjugacy separable group $G$ with a finitely presented, residually finite subgroup $H$ such that the inclusion $H\hookrightarrow G$ induces an isomorphism on profinite completions but the conjugacy problem is unsolvable in $H$.
\end{thmA}

Note that a finitely presented, conjugacy separable group has solvable conjugacy problem, just as residual finiteness implies a solution to the word problem.  Therefore, the subgroup $H$ is not conjugacy separable.

\smallskip

Our next application concerns certain one-relator groups.  Words in elements of $S \sqcup S^{-1}$ for a set $S$ can be regarded as elements of the free group $F$ on $S$. Such a word $w$ is \emph{positive} if it is in fact a word in elements of $S$. We refer the reader to page 240 of \cite{lyndon_combinatorial_1977} for the definition of the small-cancellation condition $C'(1/6)$.  The $C'(1/6)$ condition on $w$ is generic, in a suitable sense, among all positive words (\cite{Wise3}, Theorem 6.1).  Thus the following theorem implies that a generic positive one-relator group is conjugacy separable.

\begin{thmA}\label{t: One-relator groups}
If $w\in F$ is a positive $C'(1/6)$ element that is not a proper power then the one-relator group $F/\langle\langle w\rangle\rangle$ is conjugacy separable.
\end{thmA}

Here is a brief outline of the structure of this paper.  Definitions and preliminary material are introduced in Section \ref{s: Background}, and in Section \ref{s: Reduction to clean} we reduce to the clean case. The proof of Theorem \ref{t: Main Theorem} in the clean case is contained in Section \ref{s: Proof of clean}.  Theorems \ref{t: Rips construction}, \ref{t: Pair} and \ref{t: One-relator groups} are deduced in Section \ref{s: Applications}.

\subsection*{Acknowledgements}The authors would like to thank Luis Ribes for making available the manuscript of his unpublished book with Zalesskii, \emph{Profinite Trees}.  Thanks also to Martin Bridson, who initiated this work by asking whether Theorem \ref{t: Rips construction} was true.

\section{Background}\label{s: Background}

\subsection*{Notation}

We adopt the convention that if $x,y$ are group elements then $x^{y}$ denotes $yxy^{-1}$. Likewise if $A$ is a subset then $A^{y}$ denotes $yAy^{-1}$.
\smallskip

Recall that the \emph{normal core} of a subgroup $H$ in a group $G$ is the largest normal subgroup of $G$ contained in $H$. Equivalently it may be considered as the intersection of all of the conjugates of $H$.  If $H$ is of finite index then so is the normal core of $H$.

In a similar vein, given a subgroup $H$ of index $n<\infty$ in a group $G$, we define the \emph{characteristic core} of $H$ to be the intersection of all subgroups of $G$ of index $n$. Any automorphism of $G$ will preserve the set of subgroups of index $n$, and hence this characteristic core. Thus it is a characteristic subgroup (\emph{i.e.} one invariant under all automorphisms). Observe that when $G$ is finitely generated it has only finitely many subgroups of index $n$, so the characteristic core is of finite index. Of course this definition is not really dependent on the subgroup $H$, but only on its index in $G$. Nonetheless we use it as a mechanism for finding a finite-index characteristic subgroup of the ambient group inside a given subgroup.

\subsection{Separability conditions and profinite completions}
Let $G$ be a group. There are a number of properties of the group which relate to, or are directly detectable in, its finite quotients. We give here a brief discussion of some of these properties, insofar as they relate to our work. Where multiple equivalent definitions are given without proof, their equivalence is a standard result, and the different versions are stated to help lend the reader perspective.

\begin{definition}
One says $G$ is \emph{residually finite} if given $g\in G\smallsetminus 1$ there exists a finite quotient $\phi(G)$ with $\phi(g) \neq 1$.
\end{definition}

This is equivalent to the statement that the intersection of all the finite-index normal subgroups of $G$ is trivial, for such subgroups not containing $g$ correspond precisely to quotients of the above type.

\begin{definition}
A subset $S$ of $G$ is \emph{separable} if given $g \in G \smallsetminus S$ there exists a finite quotient $\phi(G)$ with $\phi(g)\notin \phi(S)$.
\end{definition}

Being residually finite is thus the same as having the trivial subgroup separable.

\begin{definition}
One says $G$ is \emph{conjugacy separable} if every conjugacy class of $G$ is separable.
\end{definition}

Equivalently for any two non-conjugate elements of $G$ there is some finite quotient of $G$ where their images remain non-conjugate. This is a stronger property than being residually finite.

\begin{definition}
A subgroup $H$ of $G$ is \emph{conjugacy distinguished} if, for any $g\in G$ which is not conjugate into $H$, there exists a finite quotient $\phi(G)$ in which $\phi(g)$ is not conjugate into $\phi(H)$.
\end{definition}

Being conjugacy distinguished is a stronger property for a subgroup than being separable.

There is a natural topology on $G$ that puts these definitions into context.

\begin{definition}
The \emph{profinite topology} on $G$ is the coarsest topology making every homomorphism from $G$ to a finite group (with the discrete topology) continuous.
\end{definition}

Many of the previous definitions can be restated in this language. For instance, being separable is equivalent to being closed in this topology. Observe that a basis for the profinite topology is formed by the preimages of singleton elements of the finite quotients; these are precisely the cosets of the finite-index normal subgroups of $G$.

We can frequently see these properties in the profinite completion of a group.

\begin{definition}
Let $G$ be a group. Consider the set of finite quotients $\phi:G \rightarrow A$; these naturally form an inverse system (if $\ker(\phi_i) \subset \ker(\phi_j)$ then $\phi_i(G)\twoheadrightarrow \phi_j(G)$). The \emph{profinite completion} $\hat{G}$ of $G$ is the inverse limit of this system.
\end{definition}

There is a natural homomorphism $G \rightarrow \hat{G}$, the kernel of which is the intersection of all finite-index normal subgroups of $G$. Thus this natural map is an injection precisely when $G$ is residually finite. In the cases of interest to us, $G$ will be residually finite and so inject into its profinite completion; we sometimes identify $G$ with its image under this injection.

For a residually finite group $G$, we may state the criterion for conjugacy separability in this language: if two elements of $G$ are conjugage in $\hat{G}$ then they are already conjugate in $G$.

\begin{lemma}For a residually finite group $G$ with finite-index subgroup $H$, $\hat{G}=G\hat{H}$.
\end{lemma}
\begin{proof}
An element of $\hat{G}$ is a consistent choice of elements in the inverse system of finite quotients of $G$. The family of quotients given by $\{I: I\unlhd H \& I\unlhd G\}$ is cofinal in this system (and also in the corresponding system for $\hat{H}$), so we may instead just consider these. Note that there is a smallest such quotient, $\phi:G\rightarrow G/N$, where $N$ is the normal core of $H$ in $G$. Given $\gamma\in\hat{G}$, we can find $d\in G$ such that $\phi(d)=\phi(\gamma)$. Then $\phi(d^{-1}\gamma)=1$, so the consistency forces that in every quotient the image of $d^{-1}\gamma$ lies in the image of $H$, which is to say that $d^{-1}\gamma\in\hat{H}$. Therefore $\gamma=d\gamma'$, for some $\gamma'\in\hat{H}$.
\end{proof}

\subsection{Graphs of groups and profinite graphs of groups}

We assume that the reader has some familiarity with the theory of groups acting on trees and in particular \emph{graphs of groups} in the sense described by Serre in \cite{Serre}. When we say that $\mathcal{G}$ is a graph of groups we mean that it contains the data of the underlying graph, the edge groups, the vertex groups and the attaching maps.  We denote the set of vertices of the underlying graph by $V(\mathcal{G})$ and likewise the set of edges by $E(\mathcal{G})$.  The vertex group corresponding to $v\in V(\mathcal{G})$ is denoted $\mathcal{G}_v$, and similarly $\mathcal{G}_e$ is the edge group for an edge $e\in E(\mathcal{G})$.  We will often abuse notation and identify $\mathcal{G}_e$ with its image under an attaching map.

The fundamental group of the graph of groups, $G=\pi_{1} (\mathcal{G})$, acts on the associated Bass--Serre tree $T$, which action recovers the underlying graph of $\mathcal{G}$ as its quotient. For a vertex $v$ and edge $e$ in $T$ we denote the corresponding vertex and edge stabilizers by $G_v$ and $G_e$ respectively.

\begin{definition}
A graph of groups $\mathcal{G}$ is \emph{clean} if each vertex group is a finitely generated free group, and each image of each edge group under an attaching map is a free factor in the relevant vertex group.
\end{definition}

Note that our definition of clean is not quite the same as the definition in \cite{Wise2}.  Wise does not require the vertex groups to be finitely generated.

\begin{definition}
A graph of groups $\mathcal{G}$  with fundamental group $G$ is \emph{$k$-acylindrical} if given any path of ($k+1$) edges $e_{0}e_{1} \ldots e_{k}$ in the associated tree $T$, the intersection $\bigcap_{i=0}^{k}G_{e_{i}}$ of the stabilizers is trivial.
\end{definition}
We will generally be interested in the case $k=1$. This means that the intersection of any two adjacent edge stabilizers is trivial. Equivalently given any vertex group $\mathcal{G}_{v}$ the set $E$ of edge groups in the vertex group is \emph{malnormal}, \emph{i.e.} if $A,B\in E, g \in \mathcal{G}_{v}$, then either $A\cap B^{g}$ is trivial or $A=B$ and $g \in A$.

\begin{definition}
A graph of groups $\mathcal{G}$ has \emph{efficient profinite topology} if $G=\pi_{1} (\mathcal{G})$ is residually finite, the profinite topology on $G$ induces the profinite topology on each vertex and each edge group (viewed as subgroups of $G$), and each vertex and each edge group is separable in $G$.
\end{definition}

There is a similar notion of profinite groups acting on profinite trees. We recall some basic definitions and facts.  For details, see \cite{profinitetrees} or \cite{zalesskii_fundamental_1989}. Recall that a \emph{profinite tree} $T$ is an inverse limit of connected finite graphs such that $H_{1}(T, \hat{\mathbb{Z}})=0$. If a profinite group $G$ acts continuously on $T$ with finite quotient $\Gamma$ then there is a profinite graph of groups $\mathcal{G}$ over $\Gamma$ where the vertex and edge groups are isomorphic to the stabilizers of the corresponding vertices and edges in $T$, and the profinite fundamental group $\Pi_1(\mathcal{G})$ is isomorphic to $G$. Similarly, associated to a finite graph of profinite groups is a profinite tree on which the fundamental group acts. As in the general (non-profinite) case, each subgroup has a minimal invariant subtree (see Lemma 2.2 of \cite{RZ}).  When $\mathcal{G}$ is a finite graph of residually finite groups, we can consider the graph of profinite groups $\hat{\mathcal{G}}$ whose edge and vertex groups are the profinite completions of those of $\mathcal{G}$. In this case $\widehat{\pi_1(\mathcal{G})}$ is isomorphic to $\Pi_1(\hat{\mathcal{G}})$.

\begin{definition}
A graph of groups $\mathcal{G}$  is \emph{profinitely $k$-acylindrical} if given any path of ($k+1$) edges $e_{0}e_{1} \ldots e_{k}$ in the profinite tree $\hat{T}$ associated with $\hat{\mathcal{G}}$, the intersection $\bigcap_{i=0}^{k}\hat{G}_{e_{i}}$ of the stabilizers is trivial, where $\hat{G}=\Pi_1(\hat{\mathcal{G}})$.
\end{definition}

\section{Reduction to the clean case}\label{s: Reduction to clean}

As remarked in the introduction, conjugacy separability is not always a commensurability invariant \cite{goryaga_example_1986}.  The next lemma provides a condition under which we can deduce that conjugacy separability of a finite-index subgroup implies conjugacy separability of the whole group.  We will say that a group $G$ has \emph{unique roots} if, whenever $a^n=b^n$ for $a,b\in G$ and $n\in\Z\smallsetminus 0$, it follows that $a=b$.

\begin{lemma}\label{separabilitypasses}
If $G$ is a finitely generated group with unique roots and $G$ has a finite-index subgroup $H$ which is conjugacy separable, then $G$ is conjugacy separable.
\end{lemma}
\begin{proof}
Let $g_{1},g_{2}\in G$ be conjugate in every finite quotient, \emph{i.e.} there exists $\gamma\in\hat{G}$ such that $g_{1}=g_{2}^{\gamma}$. We need to demonstrate that $g_{1}$ and $g_{2}$ are already conjugate in $G$.

Observe that as the index of $H$ in $G$ is finite (equal to $m$, say), there is some positive integer $n$ such that $g^{n}\in H$ for every $g\in G$ (we could take $n=m!$, since for every $g$, $g^{k}\in H$ for some $k\leq m$). Now $g_{1}=g_{2}^{\gamma}$, so $g_{1}^{n}=(g_{2}^{\gamma})^n=(g_{2}^{n})^{\gamma}$.

Since $\hat{G}=G\hat{H}$, we may take $\gamma=d\gamma'$, for some $d\in G,\gamma'\in\hat{H}$. We have $g_{1}^{n}=(g_{2}^{n})^{d\gamma'}=((g_{2}^{d})^{n})^{\gamma'}$. All $n^{th}$-powers of elements of $G$ are in $H$, so we have two elements of $H$ which are conjugate in $\hat{H}$. But we know that $H$ is conjugacy separable, so they are already conjugate in $H$: $g_{1}^{n}=((g_{2}^{d})^{n})^h$ for some $h\in H$. It follows that $g_{1}^{n}=(g_{2}^{dh})^n$, and, by our assumption about unique roots in $G$, that $g_{1}=g_{2}^{dh}$.
\end{proof}

It is a standard fact that torsion-free hyperbolic groups have the unique roots property.

\begin{lemma}\label{l: hyperbolic unique roots}
If $G$ is a torsion-free hyperbolic group and $a^n=b^n$ for $a,b\in G$, with $n$ a positive integer, then $a=b$.
\end{lemma}
\begin{proof}
The case $a=1$ is trivial. Let $\gamma=a^n \neq 1$. We consider the centralizer $Z(\gamma)$. As $G$ is torsion-free, $\gamma$ is non-trivial and has infinite order. Corollary 7.2 of \cite{CDP} states that the centralizer of any element of infinite order in a hyperbolic group is virtually cyclic. So $Z(\gamma)$ is torsion free and virtually cyclic, so cyclic ($=\langle q \rangle$, say). As $a$ and $b$ lie in $Z(\gamma)$, each is a power of $q$. But $a^n=b^n$, so we have $a=b$.
\end{proof}

In fact, finitely generated 1-acylindrical graphs of free groups are hyperbolic by the Combination Theorem \cite{bestvina_combination_1992}.   However, we will give a simple, direct proof that the groups in question have unique roots using Bass--Serre theory.

\begin{lemma}\label{uniqueroots}
If $\mathcal{G}$ is a $1$-acylindrical graph of groups and every vertex group has unique roots then $G=\pi_1(\mathcal{G})$ has unique roots.
\end{lemma}
\begin{proof}
Let $a,b\in G$ and suppose that $a^n=b^n$ for $n>0$.  Let $T$ be the Bass--Serre tree of $\mathcal{G}$.  Then $a$ and $b$ have the same translation length on $T$, and in particular are either both hyperbolic or both elliptic.  If they are both hyperbolic then they have a common axis which they translate an equal distance in the same direction, so $a^{-1}b$ fixes the whole axis.  Because $\G$ is $1$-acylindrical, it follows that $a=b$.

Suppose $a$ and $b$ are both elliptic.  Let $u$ be a vertex stabilized by $a$ and $v$ a vertex stabilized by $b$.  The claim is that $a$ also stabilizes $v$.  Suppose therefore that $u\neq v$.  Note that $a^n=b^n$ stabilizes both $u$ and $v$.  The fundamental group of a graph of torsion-free groups is torsion free, so $a^n$ is non-trivial and therefore $u$ and $v$ are adjacent because $\mathcal{G}$ is 1-acylindrical.  So $a^n$ stabilizes the intermediate edge $e$.  But $a^n$ also stabilizes $ae$, so because the fixed point set of $a^n$ has diameter 1 it follows that $a$ stabilizes $e$ and hence $v$.

The result now follows from the hypothesis that the vertex groups of $\mathcal{G}$ have unique roots.
\end{proof}

Free groups have unique roots (for example by either of Lemmas \ref{l: hyperbolic unique roots} or \ref{uniqueroots}, or direct observation), and so the fundamental group of a 1-acylindrical graphs of free groups has unique roots.

The hypotheses of Theorem \ref{t: Main Theorem} only assume that the edge groups are finitely generated.  The next observation enables us to assume that the vertex groups are also finitely generated.

\begin{remark}\label{r: Fg vertex groups}
If $\mathcal{G}$ is a finite graph of free groups with finitely generated edge groups then each vertex group has a finitely generated free factor that contains the image of every edge map.  Replacing each vertex group with this free factor, we obtain a new finite graph of groups $\mathcal{G}'$ with finitely generated vertex groups.  Now $\pi_1(\mathcal{G})=\pi_1(\mathcal{G}')*F$ where $F$ is a (perhaps infinitely generated) free group.
\end{remark}

The main technical result of \cite{Wise2} is Theorem 11.3, which asserts that the 1-acylindrical graphs that we consider are virtually clean.  (In fact, Wise's hypotheses are slightly weaker---he assumes that his graph of groups is \emph{thin}.)

\begin{thm}[Wise]\label{t: Wise main theorem}
Let $\mathcal{G}$ be a finite, 1-acylindrical graph of finitely generated free groups with finitely generated edge groups.  Then $\pi_1(\mathcal{G})$ has a subgroup of finite index $H$ such that the induced graph-of-groups decomposition for $H$ is 1-acylindrical and clean.
\end{thm}

Combining the results of this section, it is enough to prove Theorem \ref{t: Main Theorem} for clean graphs of finitely generated groups.  In the next section, we will prove the following.

\begin{prop}\label{p: Clean case}
Let $\mathcal{C}$ be a finite, clean, 1-acylindrical graph of free groups with finitely generated edge groups.  Then $C=\pi_1(\mathcal{C})$ is conjugacy separable.
\end{prop}

\begin{proof}[Proof of Theorem \ref{t: Main Theorem}]
Let $\mathcal{G}$ be a finite, 1-acylindrical graph of free groups with finitely generated edge groups.  The free product of two conjugacy separable groups is itself conjugacy separable \cite{Stebe}. As all free groups are conjugacy separable, Remark \ref{r: Fg vertex groups} reduces us to the case in which the vertex groups of $\mathcal{G}$ are finitely generated.

Lemma \ref{uniqueroots} implies that $\pi_1(\mathcal{G})$ has unique roots.  Therefore, by Lemma \ref{separabilitypasses}, it is enough to prove that a finite-index subgroup of $\pi_1(\mathcal{G})$ is conjugacy separable.  Proposition \ref{p: Clean case} asserts that the finite-index subgroup of Theorem \ref{t: Wise main theorem} is conjugacy separable, which completes the proof.
\end{proof}

\section{The proof of the clean case}\label{s: Proof of clean}

In light of the last section, to prove Theorem \ref{t: Main Theorem} it remains only to prove Proposition \ref{p: Clean case}.  The following theorem (Theorem 5.2 of \cite{WtZ}) gives conditions for conjugacy separability.

\begin{thm}\label{conditions}
Let $\mathcal{G}$ be a finite graph of groups with conjugacy separable vertex groups, and let $G=\pi_{1}(\mathcal{G})$. Suppose that the profinite topology on $\mathcal{G}$ is efficient and that $\mathcal{G}$ is profinitely 2-acylindrical. For any vertex $v$ of the underlying graph, and incident edges $e$ and $f$, suppose furthermore that the following conditions hold:
\begin{enumerate}
\item for any $g\in \mathcal{G}_{v}$ the double coset $\mathcal{G}_{e}g\mathcal{G}_{f}$ is separable in $\mathcal{G}_{v}$;
\item the edge group $\mathcal{G}_{e}$ is conjugacy distinguished in $\mathcal{G}_{v}$;
\item the intersection of the closures of $\mathcal{G}_{e}$ and $\mathcal{G}_{f}$ in the profinite completion of $\mathcal{G}_{v}$ is equal to the profinite completion of their intersection, \emph{i.e.} $\bar{\mathcal{G}}_{e}\cap\bar{\mathcal{G}}_{f}=\widehat{\mathcal{G}_{e}\cap \mathcal{G}_{f}}$.
\end{enumerate}
Then $G$ is conjugacy separable.
\end{thm}

To prove Proposition \ref{p: Clean case}, we shall show that the hypotheses of Theorem \ref{conditions} apply.   Noting that the vertex groups of $\mathcal{C}$ are all free (as they are finite-index subgroups of a finitely generated free group), Lemma \ref{cleanefficient} gives the efficiency of the topology on $\mathcal{C}$.  Lemma \ref{acylindrical} tells us that $\mathcal{C}$ is profinitely 1-acylindrical, which is stronger than being profinitely 2-acylindrical.  We apply Lemmas \ref{doublecoset}, \ref{conjugacydistinguished} and \ref{intersection} to give conditions (1), (2) and (3) of Theorem \ref{conditions} respectively. Thus Theorem \ref{conditions} implies that $C=\pi_1(\mathcal{C})$ is conjugacy separable. Combining this with Lemma \ref{separabilitypasses} proves Proposition \ref{p: Clean case} and hence Theorem \ref{t: Main Theorem}.

We start by proving efficiency.

\begin{lemma}\label{cleanefficient}
If $\mathcal{G}$ is a finite, clean graph of groups then $G=\pi_{1}(\mathcal{G})$ has efficient profinite topology.
\end{lemma}
\begin{proof}
Note first that by \cite{Wise2}, $G$ is residually finite.

There are two remaining things we must show: that the profinite topology on $\mathcal{G}$ induces the full profinite topology on each of its vertex groups and edge groups, and that each vertex group and each edge group is separable in $\mathcal{G}$. But as $\mathcal{G}$ is clean the edge groups are free factors in free vertex groups, so the profinite topology at a vertex group induces the full profinite topology on each of its edge groups, and the edge groups are separable in the vertex groups. Thus we are reduced to showing that the full topology is induced at each vertex group, and that each vertex group is separable in $\mathcal{G}$; the corresponding facts for edge groups follow.

Let $v$ be a vertex of $\mathcal{G}$ and $N_{v}$ a finite-index normal subgroup of $\mathcal{G}_{v}$. We construct a finite-index normal subgroup $N$ of $G$ with $N\cap \mathcal{G}_{v}\subset N_{v}$. This will prove the first of the two remaining criteria. Let $F=*_{v\in V(\mathcal{G})}\mathcal{G}_{v}$, the free product of all of the vertex groups. By a theorem of Marshall Hall \cite{Hall} there is some finite-index $B\leq F$ of which $N_{v}$ is a free factor, so we must have that $B \cap \mathcal{G}_{v}\subset N_{v}$. Now let $B'$ be the characteristic core of $B$ in $F$ (indeed one can take $B'$ to be any finite-index characteristic subgroup of $F$ lying inside $B$).

Then by intersection this induces a characteristic subgroup $B_u=B'\cap \mathcal{G}_u$ at each vertex group $\mathcal{G}_u$. For if $\sigma$ is an automorphism of $\mathcal{G}_u$, then as $\mathcal{G}_u$ is a free factor of $F$ we may extend $\sigma$ to an automorphism $\Sigma$ of $F$ by setting it to be the identity on all the other generators.  As $B'$ is characteristic in $F$, so is invariant under $\Sigma$, it follows that $B_u$ is invariant under $\sigma$. The edge groups are free factors in the vertex groups, so by the same argument intersection with the characteristic subgroup in any vertex group gives a characteristic subgroup in each edge group. But each edge group $\mathcal{G}_e$ injects into two separate vertex groups, $\mathcal{G}_{u_{e,1}}$ and $\mathcal{G}_{u_{e,2}}$, each of which induces a characteristic subgroup inside the edge group. We show that these two are the same, \emph{i.e} that $\mathcal{G}_e\cap B_{u_{e,1}}=\mathcal{G}_e\cap B_{u_{e,2}}$. Because the edge groups are free factors in the vertex groups, $\mathcal{G}_e$ is embedded as a free factor of $F$ of the same rank via either of the adjoining vertex groups, so there is an automorphism of $F$ sending the image of one of these embeddings to the other. But because $B'$ is characteristic it is invariant under all automorphisms of $F$, so the induced edge subgroups $B_e=B'\cap \mathcal{G}_e$ are well-defined.

From here, we pass to a quotient $\Phi:G\rightarrow G'=\pi_{1}(\mathcal{G}')$ by quotienting out each vertex group $\mathcal{G}_u$ by $B_u$, and each edge group $\mathcal{G}_e$ by $B_e$. As $B_e=\mathcal{G}_e\cap B_u$ for any vertex $u$ adjoining an edge $e$, the attaching maps of $\mathcal{G}'$ are still injections. Note that $\mathcal{G}'$ has the same underlying graph as $\mathcal{G}$.

This is a finite graph of finite groups, so $G^{\prime}$ is virtually free, and hence residually finite. Thus we can find a finite-index subgroup $H\lhd G^{\prime}$ which misses everything in the image $\Phi(\mathcal{G}_{v})$. By composing $\Phi$ with quotienting by $H$, we get a homomorphism from $G$ to a finite group. We take the kernel of this homomorphism to be $N$; it has precisely the properties we wanted.

\smallskip
We must now prove the second of the two remaining criteria for efficiency: that the vertex groups are separable (\emph{i.e.} that each vertex group in $G$ is an intersection of finite-index subgroups of $G$). For this we follow a similar construction. Given a vertex group $\mathcal{G}_{v}$ and $g\in G\smallsetminus \mathcal{G}_{v}$ we want to find a finite-index normal subgroup $N$ of $G$ with $\mathcal{G}_v \leq N$ and $g\notin N$. Consider $G$ as the fundamental group of the graph of groups in the sense of \cite{Serre} with a basepoint at $v$ (this cavalier attitude to basepoints is here permissible because changing basepoints applies an automorphism to the group, but we are just interested in finding the existence of a subgroup, so we are free to undo the automorphism at the end). Now we take a reduced form $g=g_{0}e_{1}g_{1}\ldots e_{k}g_{k}$, which is non-trivial. Whenever we have $e_{i}=\bar{e}_{i+1}$, $g_{i}\notin \mathcal{G}_{e_{i}}$. Let $K_{i}\leq \mathcal{G}_{v_{i}}$ be finite-index such that $g_{i}\notin K_{i}$ and $\mathcal{G}_{e_{i}}<K_{i}$; such exists because $\mathcal{G}_{e_{i}}$ is a free factor in $\mathcal{G}_{v_{i}}$. Now let $B_{i}$ be a finite-index subgroup of $F$ with $B_{i}\cap \mathcal{G}_{v_{i}}\subset K_{i}$ (where $F$ is, as above, the free product of all of the vertex groups). Take $D$ to be intersection of all of the $B_{i}$ (in the event that there are no values of $i$ to consider, take $D=D^{\prime}=F$), and $D'$ to be some finite-index characteristic subgroup inside $D$. Then, as before, we take a quotient $G\rightarrow G^{\prime}$ where at each vertex group we quotient out the intersection of that group with $D'$. Then the image of the reduced form for $g$ is still reduced, as our construction specifically avoids the cases where it might collapse; in particular the image of $g$ is non-trivial. Again $G^{\prime}$ is virtually free and hence residually finite, so by composition we may find a finite quotient of $G$ in which the image of $g$ is distinct from the image of $\mathcal{G}_{v}$; we take $N$ to be the preimage of the image of $\mathcal{G}_v$.
\end{proof}

\begin{lemma}\label{acylindrical}
The graph of groups $\mathcal{C}$ is profinitely 1-acylindrical, \emph{i.e.} for adjacent edges $e$ and $f$ in the profinite tree associated with $\hat{C}$, the intersection of the edge stabilizers, $\hat{C}_{e}\cap\hat{C}_{f}$, is trivial.
\end{lemma}
\begin{proof}
This is equivalent to saying that for every vertex $v\in V(\C)$ and incident edges $e,f\in E(\C)$ (not necessarily distinct), whenever $\bar{\C}_e\cap\bar{\C}_f^{\gamma}\neq\{1\}$  (closures being taken in the profinite vertex group $\hat{\C}_v$) for $\gamma\in\hat{C}_v$, we have $\C_e=\C_f$ and $\gamma\in\bar{\C}_e$.  Therefore let $\gamma \in \hat{\C}_v$ be such that $\bar{\C}_e\cap\bar{\C}_f^{\gamma}$ contains a non-trivial element $\delta$.

In what follows, for simplicity of notation we will write $F=\C_v$, $A=\C_e$ and $B=\C_f$.  The profinite free group $\hat{F}$ acts freely on its profinite Cayley tree $\hat{T}$. The closure of the edge group $\bar{A}$ has a minimal invariant subtree $\hat{T}_{\bar{A}}$, with the property that $T_{A}/A=\hat{T}_{\bar{A}}/\bar{A}$ (Lemma 2.2 of \cite{RZ}). There is a similar invariant subtree $\hat{T}_{\bar{B}}$, and of course $\gamma\hat{T}_{\bar{B}}$ is the minimal invariant subtree of $\bar{B}^\gamma$. Let $\hat{e}$ be an edge of the minimal invariant subtree of $\langle\delta\rangle$.  Then $\hat{e}$ is contained in $\hat{T}_{\bar{A}}$ and $\gamma\hat{T}_{\bar{B}}$.  Hence there is an edge $e$ of $T_A$ and an element $\alpha$ of $\bar{A}$ such that $\hat{e}=\alpha e$.  Similarly, there is an edge $f$ of $T_B$ and an element $\beta$ of $\bar{B}$ such that $\hat{e}=\gamma\beta f$.

Therefore, $\beta^{-1}\gamma^{-1}\alpha e=f$.  But because $\hat{T}/\hat{F}=T/F$, there is an element $g$ of $F$ such that $ge=f$. As $\mathcal{C}$ is 1-acylindrical and clean, $A$ and $B$ are free factors of $F$ and are either equal or have trivial intersection; it follows that $A=B$ and $g\in A$.  Therefore, using the freeness of the action of $\hat{F}$, $\gamma^{-1} = \beta g\alpha^{-1}$ which is in $\bar{A}$, as required.
\end{proof}

To show that the double cosets are separable in the vertex groups, we appeal to the fact that the vertex groups are finitely generated free groups. The following lemma was originally due to Gitik and Rips, and is proved in a more general form by Niblo in \cite{Niblo}.

\begin{lemma}\label{doublecoset}
For a finitely generated free group $F$, any subgroups $H_{1},H_{2}<F$, and any $g\in F$, the double coset $H_{1}gH_{2}$ is separable in $F$.
\end{lemma}

Likewise the requirement that edge groups are conjugacy distinguished in vertex groups is taken care of by the fact that the vertex groups are free. The following lemma is Proposition 2.5 in a paper by Wilson and Zalesskii \cite{WsZ}.
\begin{lemma}\label{conjugacydistinguished}
Every finitely generated subgroup of a finitely generated virtually free group is conjugacy distinguished.
\end{lemma}

\begin{lemma}\label{intersection}
For all finitely generated subgroups $A,B$ of a finitely generated free group $F$, $\bar{A}\cap\bar{B}=\widehat{A\cap B}$ (closure taken in $\hat{F}$).
\end{lemma}
\begin{proof}
Proposition 2.4 in \cite{WsZ} gives us that $\bar{A}\cap\bar{B}=\overline{A\cap B}$. But we are in a free group, so $A\cap B$ is separable, and $\overline{A\cap B}=\widehat{A\cap B}$.
\end{proof}

This completes the proof of Proposition \ref{p: Clean case} and so of Theorem \ref{t: Main Theorem}.

\section{Applications}\label{s: Applications}

\subsection{The Rips construction}

In \cite{Wise1}, Wise provided a residually finite version of the Rips construction.  To prove Theorem \ref{t: Rips construction}, we recall the details of his construction here.  Given a group $Q$ presented by $\langle a_{1},\ldots,a_{r}\mid R_{1},\ldots,R_{s}\rangle$, we take $\Gamma$ to be the group presented by
 \begin{equation}
    \left\langle \parbox[c]{2cm}{$a_{1},\ldots,a_{r},\\x,y,t$} \hspace*{0.1cm} \vline \hspace*{0.1cm} \parbox[c]{8.5cm}{$R_{j}=Y_{j}tZ_{j}t^{-1}\hspace*{3.5cm}\rbrace \hspace*{0.1cm}1\leq j \leq s \\ \parbox[c]{6cm}{$x^{a_{i}}=A_{i}tB_{i}t^{-1} \hspace*{0.4cm} x^{a_{i}^{-1}}=D_{i}tE_{i}t^{-1}\\ y^{a_{i}}=J_{i}tK_{i}t^{-1} \hspace*{0.4cm} y^{a_{i}^{-1}}=L_{i}tM_{i}t^{-1} \\ t^{a_{i}}=S_{i}tT_{i} \hspace*{1cm} t^{a_{i}^{-1}}=U_{i}tV_{i}$}\hspace*{0.1cm} \rbrace \hspace*{0.1cm} 1\leq i \leq r$}\right\rangle,
 \end{equation}
where $\{A_{i},B_{i},D_{i},E_{i},J_{i},K_{i},L_{i},M_{i},S_{i},T_{i},U_{i},V_{i},Y_{j},Z_{j}:1\leq i\leq r,1\leq j\leq s\}$ is a set of distinct, freely reduced words in the letters $x^{\pm1},y^{\pm1}$ satisfying a $C'(1/6)$ small-cancellation condition. We take $N$ to be $\langle x,y,t \rangle$; this is normal as the relations not involving the $R_{j}$ precisely force it to be so. Note that $\Gamma/N\cong Q$ because setting the generators of $N$ equal to $1$ in the above presentation recovers our  original presentation for $Q$.

Now consider the same presentation, but with the relators cyclically permuted:
 \begin{equation}
    \left\langle \parbox[c]{2cm}{$a_{1},\ldots,a_{r},\\x,y,t$} \hspace*{0.1cm} \vline \hspace*{0.1cm} \parbox[c]{8.5cm}{$Z_{j}^{t}=Y_{j}^{-1}R_{j}\hspace*{3.5cm}\rbrace \hspace*{0.1cm}1\leq j \leq s \\ \parbox[c]{6cm}{$B_{i}^{t}=A_{i}^{-1}x^{a_{i}} \hspace*{0.7cm} E_{i}^{t}=D_{i}^{-1}x^{a_{i}^{-1}}\\ K_{i}^{t}=J_{i}^{-1}y^{a_{i}} \hspace*{0.7cm} M_{i}^{t}=L_{i}^{-1}y^{a_{i}^{-1}} \\ (T_{i}a_{i})^{t}=S_{i}^{-1}a_{i} \hspace*{0.2cm} (a_{i}V_{i}^{-1})^t=a_{i}U_{i}$}\hspace*{0.1cm} \rbrace \hspace*{0.1cm} 1\leq i \leq r$}\right\rangle.
 \end{equation}
Consider the collection of words on the left hand sides of these relators, before conjugation by $t$. These differ by at most one letter from members of the set of words with the small-cancellation condition above. Thus by small-cancellation theory they freely generate a subgroup of the free group on $\{a_{1},\ldots,a_{r},x,y,t\}$. Likewise the collection of words on the right hand sides freely generate a subgroup. Thus $\Gamma$ splits as a single HNN-extension of a finitely generated free group, with stable letter $t$. The small-cancellation condition further ensures that it is malnormal, by a result in \cite{Wise3}.  Equivalently, the decomposition of $\Gamma$ as an HNN extension is 1-acylindrical.  That $\Gamma$ is conjugacy separable is now an immediate application of Theorem \ref{t: Main Theorem}.  This completes the proof of Theorem \ref{t: Rips construction}.

\begin{remark}
As this work was being written up, the authors became aware of the paper \cite{minasyan_hereditary_2009} by Minasyan.  One can combine this with work of Haglund and Wise \cite{haglund_special_2008} to provide an alternative proof of Theorem \ref{t: Rips construction}.  However, the techniques of \cite{haglund_special_2008} and hence \cite{minasyan_hereditary_2009} do not extend to arbitrary 1-acylindrical graphs of free groups.
\end{remark}

\subsection{Discrete groups with isomorphic profinite completion}

By applying our conjugacy separable version of the Rips construction to a superperfect group $Q$ with unsolvable word problem that has a finite classifying space and no finite quotients, we obtain a short exact sequence $1 \rightarrow N \rightarrow \Gamma \rightarrow Q \rightarrow 1$ with $\Gamma$ (and hence $\Gamma \times \Gamma$) conjugacy separable. Bridson shows in \cite{Bridson} that such groups $Q$ exist, and explains that the fibre product $P \overset{u}\hookrightarrow \Gamma \times \Gamma$ has an unsolvable conjugacy problem. On the other hand, Bridson and Grunewald show that the profinite completion of $u$ induces an isomorphism of profinite completions $\hat{u}:\hat{P} \rightarrow \widehat{\Gamma \times \Gamma}$ \cite{BridGrun}. A finitely presented conjugacy separable group has solvable conjugacy problem.  Theorem \ref{t: Pair} follows.

Hence we see that conjugacy separability is not a property that can be detected in the profinite completion, even for finitely presented, residually finite groups, but requires in addition information about the particular injection of the group into its profinite completion.

\subsection{One-relator groups}

Let $F$ be a finitely generated free group and let $w\in F$ be a positive element, that is an element in which only positive powers of the generators appear.  The quotient $\Gamma=F/\langle\langle w\rangle\rangle$ is called a \emph{positive one-relator group}.  Such a group $\Gamma$ is torsion-free if and only if $w$ is not a proper power.  We refer the reader to \cite{lyndon_combinatorial_1977} for the definition of the $C'(1/6)$ small-cancellation condition on $w$, but we recall that if $w$ satisfies $C'(1/6)$ then $\Gamma$ is hyperbolic \cite{Gromov}.   If $w$ is positive and satisfies $C'(1/6)$ then we say that the corresponding one-relator group $\Gamma$ is a \emph{positive $C'(1/6)$ one-relator group}.

\begin{proof}[Proof of Theorem \ref{t: One-relator groups}]
The word $w$ involves at most finitely many generators of $S$, so by \cite{Stebe} we may reduce to the case in which $F$ is finitely generated.  A finitely generated, positive $C'(1/6)$ one-relator group $\Gamma$ has a finite-index subgroup $\Gamma'$ that splits as a 1-acylindrical graph of free groups (\cite{Wise3}, Theorem 1.1).  By Theorem \ref{t: Main Theorem}, therefore, $\Gamma'$ is conjugacy separable.  But $\Gamma$ is a torsion-free hyperbolic group and so by Lemma \ref{l: hyperbolic unique roots} has unique roots.  Therefore, by Lemma \ref{separabilitypasses}, $\Gamma$ is conjugacy separable.
\end{proof}

\end{document}